\documentclass[12pt]{amsart}
\usepackage{fullpage,url,amssymb,amsmath,graphicx}

\usepackage{amsmath, amsthm, amssymb}
\usepackage{url}
\usepackage{graphicx}
\usepackage {amsmath}
\usepackage{amsthm}
\usepackage {amssymb}
\usepackage {graphicx}
\usepackage[enableskew]{youngtab}

\usepackage{amsmath}
\usepackage {amssymb}
\usepackage {graphicx,enumerate,booktabs}
\usepackage {color, tikz}
\usepackage[all]{xy}

\newtheorem{definition}{Definition}
\newtheorem{lemma}{Lemma}[section]
\newtheorem{theorem}[lemma]{Theorem}
\newtheorem*{theoremNN}{Theorem}

\newtheorem{claim*}{Claim}
\newtheorem{remark}[lemma]{Remark}

\numberwithin{equation}{section}

\newcommand{\Pic}{\operatorname{Pic}}

\newcommand{\MC}{\operatorname{MaxCut}}
\newcommand{\SD}{\operatorname{SD}}

\begin{document}
\title{The maximum cut problem on blow-ups of multiprojective spaces.}
\author{
Mauricio Junca}
\author{
Mauricio Velasco
}
\address{
Departamento de matem\'aticas\\
Universidad de los Andes\\
Carrera $1^{\rm ra}\#18A-12$\\ 
Bogot\'a, Colombia
}
\email{mj.junca20@uniandes.edu.co, mvelasco@uniandes.edu.co}
\begin{abstract}
The maximum cut problem for a quintic del Pezzo surface ${\rm Bl}_{4}(\mathbb{P}^2)$ asks: Among all partitions of the $10$ exceptional curves into two disjoint sets, what is the largest possible number of pairwise intersections?  In this article we show that the answer is twelve. More generally, we obtain bounds for the maximum cut problem for the minuscule varieties $X_{a,b,c}:={\rm Bl}_{b+c}(\mathbb{P}^{c-1})^{a-1}$ studied by Mukai and Castravet-Tevelev and show that these bounds are asymptotically sharp for infinite families. 
We prove our results by constructing embeddings of the classes of $(-1)$-divisors on these varieties which are optimal for the semidefinite relaxation of the maximum cut problem on graphs proposed by Goemans and Williamson. These results give a new optimality property of the Weyl orbits of root systems of type $A$,$D$ and $E$.
\end{abstract}
\maketitle
\section{Introduction} Let $a,b$ and $c$ be positive integers and let $T_{a,b,c}$ be the $T$-shaped tree with $a+b+c-2$ vertices shown in Figure~\ref{Tabc}. Define $X_{a,b,c}:= {\rm Bl}_{b+c}(\mathbb{P}^{c-1})^{a-1}$ to be the algebraic variety obtained by blowing up  a set of $b+c$ general points in the multiprojective space $(\mathbb{P}^{c-1})^{a-1}$. If $T_{a,b,c}$ is the Dynkin diagram of a finite root system then the varieties $X_{a,b,c}$ can be thought of as higher-dimensional generalizations of del Pezzo surfaces (obtained when $a=2,c=3$ and $1\leq b\leq 5$) and share many of their fundamental properties. The varieties $X_{a,b,c}$ have been the focus of much recent work by Mukai, Castravet-Tevelev, Serganova-Skorobogatov, Sturmfels, Xu and the second author among others. They have appeared in connection to Mukai's answer to Hilbert's 14-th problem~\cite{Mukai},\cite{CT}, have been studied because of their close relationship with homogeneous spaces~\cite{SS1},~\cite{SS2},~\cite{StVe} and because of their remarkable combinatorial commutative algebra~\cite{SX}.

The varieties $X_{a,b,c}$ contain a finite distinguished collection of codimension one subvarieties called $(-1)$-divisors. The configuration of $(-1)$-divisors plays a fundamental role in the geometry of the varieties $X_{a,b,c}$ analogous to the role played by exceptional curves on Del Pezzo surfaces~\cite{CT}. The configuration of $(-1)$-divisors is independent of the chosen $b+c$ points as long as they are sufficiently general and is captured by the following multigraph.

\begin{figure}[h]
\caption{}
\begin{center}
\includegraphics[height=1.5in]{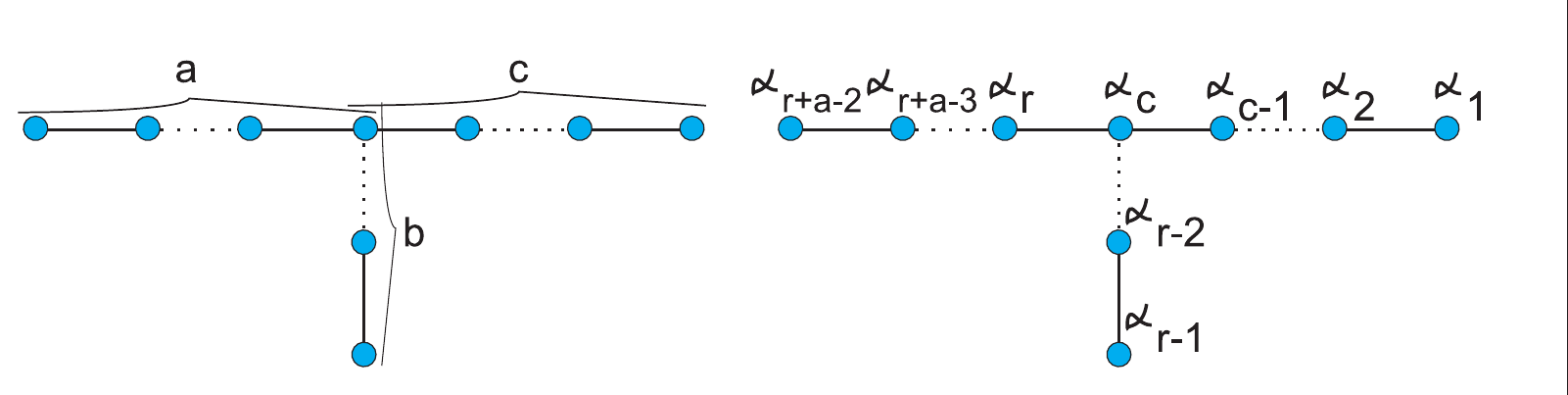}
\end{center}
\label{Tabc}
\end{figure}

\begin{definition} The multigraph of exceptional divisors $G_{a,b,c}$ has as vertices the $(-1)$ divisors and weight $M_{a,b,c}(U,V)=U\cdot V$
for distinct vertices $U$ and $V$ (see Section~\ref{geometry} for details on the construction of the product $U\cdot V$)\end{definition}
\begin{definition} The {\it maximum cut problem} for the varieties $X_{a,b,c}$ asks for the determination of the maximum cut of the multigraphs $G_{a,b,c}$. In more geometric terms it asks: among all partitions of the $(-1)$-divisors on $X_{a,b,c}$ into two sets, what is the largest possible number of pairwise intersections? 
\end{definition}

\begin{figure}[h]
\caption{}
\begin{center}
\includegraphics[height=1.65in]{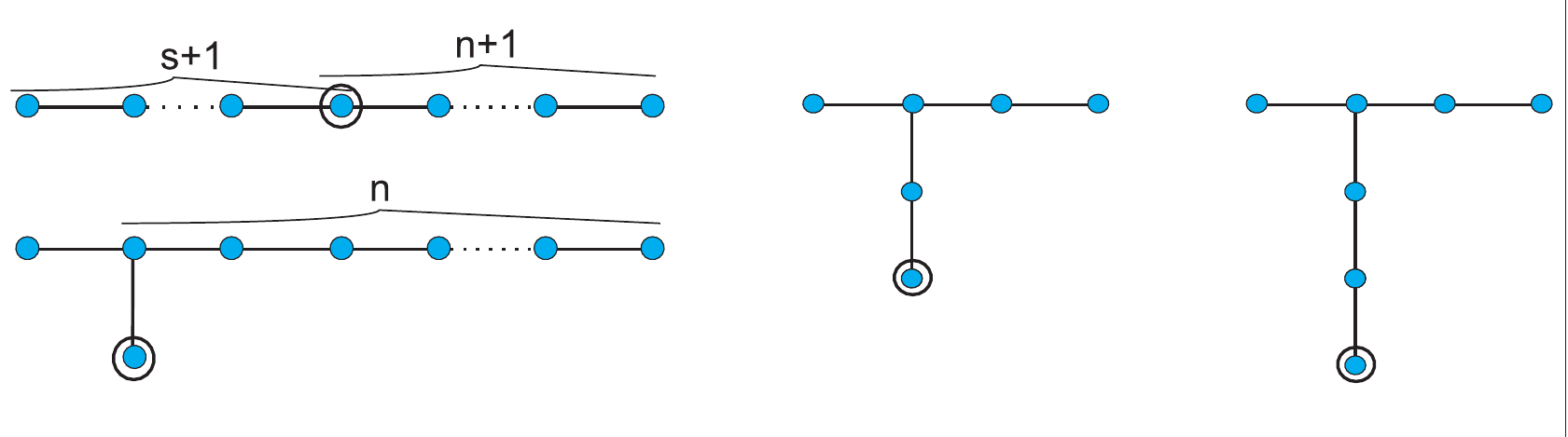}
\end{center}
\label{DD}
\end{figure}

The purpose of this article is to study the maximum cut problem on the minuscule varieties $X_{a,b,c}$ (see Figure~\ref{DD} for a list of minuscule Dynkin diagrams).  Our main result is the construction of bounds on this quantity which are asymptotically sharp for infinite families. 
\begin{theoremNN} Let $T_{a,b,c}$ be a minuscule Dynkin diagram and let $m(a,b,c)$ denote the value of the maximum cut problem for $X_{a,b,c}$.  We have $\lceil\ell(a,b,c)\rceil \leq m(a,b,c)\leq \lfloor u(a,b,c)\rfloor$ where 
\begin{tiny}
\[ 
\begin{array}{l|l|l}
Graph & \ell(a,b,c) & u(a,b,c)\\
\hline
G_{s+1,1,n+1} & \frac{1}{2\pi}\binom{r+s}{r-1}\sum_{k=0}^s \binom{s+1}{k+1}\binom{r-1}{k+1}k\arccos\left( 1-\frac{(r+s)(k+1)}{(s+1)(r-1)}\right) & \frac{r+s}{2(s+1)(r-1)} \binom{r+s}{r-1}\binom{s+1}{2}\binom{r+s-2}{r-3}\\
G_{2,2,n} & \frac{2^{r-2}}{\pi}\sum_{k=0}^{\lfloor\frac{r}{2}\rfloor} \binom{r}{2(k+1)}k\arccos\left(1-\frac{4(k+1)}{r}\right)  & (r-3)2^{2r-6}\\
G_{2,3,3} & 90 & 101.25\\
G_{2,4,3} & 516 & 560\\  
\end{array}
\]
\end{tiny}
Moreover $\lim_{|V|\rightarrow \infty} \frac{m(a,b,c)}{u(a,b,c)}=1$ so the percentage error of approximating the maxcut by its upper bound is asymptotically zero on the infinite families.
\end{theoremNN}
 We conjecture that $m(a,b,c)=u(a,b,c)$ for $G_{2,2,n}$ and $G_{n+1,1,n+1}$. In Section~\ref{Simulations} we prove the equality for $n\leq 8$ and $n\leq 5$ respectively. 

In general, determining the maximum cut of a weighted graph is a difficult problem that cannot be solved efficiently unless ${\rm P=NP}$. A more feasible alternative is to estimate this number via an approximation algorithm with a performance guarantee of $\beta$. This is a polynomial time algorithm guaranteed to produce a cut whose weight is at least a known percentage $\beta$ of the maximum cut. Probably the most important instance of a maxcut approximation algorithm is the celebrated Goemans and Williamson~\cite{GW}  stochastic approximation algorithm (henceforth GW algorithm) which has a performance guarantee $\beta\approx 87.85\%$. In this article we make a detailed analysis of the behavior of the GW algorithm on the graphs $G_{a,b,c}$ and use it as a theoretical tool to derive the above bounds.

To describe the ingredients leading to our results we briefly describe the GW algorithm on a graph $G$ (see Section~\ref{GWAlgo} for details). In the first stage the maximum cut problem is relaxed to a semidefinite optimization problem whose solution gives an ``optimal" embedding $f:V(G)\rightarrow S^{m-1}\subseteq \mathbb{R}^m$ of the graph in some sphere $S^{m-1}$. The embedding $f$ allows us to associate a cut to every hyperplane $H$ in $\mathbb{R}^m$ by splitting vertices according to the side of $H$ where their image under $f$ lies. It is known that the expected weight of a cut obtained by choosing the hyperplane $H$ uniformly at random is at least $\alpha:=\min_{0\leq\theta \leq \pi} \frac{2}{\pi}\frac{\theta }{1-\cos(\theta)}\approx 0.87856$ times the maximum cut. The second stage of the algorithm consists of uniformly sampling hyperplanes until an above average cut is reached.

\noindent
Our results rely on the following observations,
\begin{enumerate}
\item{The multigraphs $G_{a,b,c}$ are highly symmetric since they are invariant under the action of the Weyl group of the corresponding root system. Such symmetries allow us to choose our optimal embedding to be equivariant.}
\item{The geometry of the varieties provides us with a natural candidate for an equivariant embedding of the graphs $G_{a,b,c}$ in an euclidean space, namely the normalized orthogonal projection of their classes to the orthogonal complement of the canonical class of $X_{a,b,c}$. Our main result is that this embedding $f$ is optimal for the GW semidefinite relaxation. Moreover we also show that this optimal embedding may be thought of as placing the $(-1)$-divisors on the vertices of certain Coxeter matroid polytopes.}
\item{Thanks to symmetry, the dual problem of the GW semidefinite relaxation can be analyzed by understanding the spectrum of the adjacency matrix of the multigraph $G_{a,b,c}$.  
Our determination of this spectrum is the main technical tool used in proving the optimality of the embedding $f$. It relies on the following two facts,
\begin{enumerate}
\item{The concept of strongly regular multigraphs and a characterization of their spectra which we introduce in Section~\ref{srmgs}. This class of graphs contains the multigraphs $G_{a,b,c}$.}
\item{The canonical bijection between the $(-1)$-divisors on minuscule varieties $X_{a,b,c}$ and the weights of the corresponding minuscule representations (see Section~\ref{geometry} for details). These bijections allow us reduce the necessary calculations to elementary combinatorial identities.}
\end{enumerate}
}
\end{enumerate}

An important quantity in our analysis is the {\it performance ratio} $\alpha_G$ of a graph $G$, defined as the ratio of the expected weight of a random cut divided by the optimal value of the GW semidefinite relaxation (see Section~\ref{GWAlgo} for detail). This ratio satisfies $\alpha\leq \alpha_G\leq 1$ and is a one dimensional measure of the performance of the algorithm on a graph $G$, increasing as the performance of the algorithm improves.  In this article we also study how symmetry affects this quantity. We derive formulas for the performance ratio on doubly transitive graphs and are able to analyze its behavior for large classes of strongly regular graphs (see Section~\ref{GWsymmetry} for precise statements). Our main result in this direction is the following

\begin{theoremNN}  For an integer $m\geq 2$ let $\mathcal{R}(-m)$ be the collection of doubly transitive strongly regular graphs with smallest eigenvalue $-m$. The essential performance ratio $e(\mathcal{R}(-m))$ equals one, where
\[ e(\mathcal{R}(-m)) := \lim_{n\rightarrow \infty} \inf\{ \alpha_{G}: G\in \mathcal{R}(-m), |G|\geq n\}\]
\end{theoremNN}

The above Theorem says that the performance of the Goemans-Williamson algorithm improves as the size of the graphs under analysis increases approaching its theoretically possible maximum. These results can be thought of as the flip-side of worst-case performance analysis. We are no longer interested in determining the worst-case performance of an algorithm but instead we want to characterize rich classes of graphs where performance is provably better than expected. These results are especially interesting in the case of the Goemans Williamson algorithm since it is known~\cite{UGC} that either $\alpha$ is the best possible performance ratio of a certified approximation algorithm to the maximum cut problem or the Unique Games Conjecture does not hold. A possible strategy to look for counterexamples is to find classes of graphs where the performance ratio is provably better than $\alpha$. 

The material in this article is organized as follows: Section~\ref{preliminaries} contains background information on cuts and multigraphs (\S~\ref{preliminariesCombinat}), the Goemans-Williamson algorithm and performance ratios (\S~\ref{GWAlgo}.). Section~\ref{geometry} contains background material on the geometry of the varieties $X_{a,b,c}$ and a formulation of the maximum cut problem. Section~\ref{GWsymmetry} studies the role of symmetry in the GW semidefinite relaxation and proves the second Theorem above. Section~\ref{srmgs} introduces strongly regular multigraphs and characterizes their spectra. Section~\ref{maxcutMain} contains the main results of the article on the maximum cut problem for minuscule blowups of multiprojective space.

{\bf Acknowledgements.} We thank Felipe Rinc\'on and Bernd Sturmfels for helpful conversations during the completion of this work.

\section{Preliminaries}\label{preliminaries}
\subsection{Multigraphs, cuts and strongly regular graphs.} \label{preliminariesCombinat}
\begin{definition} A multigraph $G$ is a pair $(V(G),M(G))$ where $V(G)$ is a finite, totally ordered set of vertices and $M(G): V\times V\rightarrow \mathbb{R}_+$  is a nonnegative function satisfying $M(G)(v,v)=0$ and $M(G)(v,w)=M(G)(w,v)$. We can represent the function $M(G)$ via a symmetric matrix letting $M_{ij}:=M(G)(i,j)$ and thus we will refer to $M(G)$ as the adjacency matrix of $G$. Two vertices $i,j\in G$ are said to be adjacent iff $M(G)(i,j)>0$.\end{definition}
We will often drop $G$ from the notations $V(G)$, $M(G)$ when $G$ is clear from the context. Note that A graph $G$ is a multigraph whose adjacency matrix $M$ has entries in $\{0,1\}$. All our graphs are thus finite, simple, undirected, loopless graphs.
\begin{definition} The automorphism group of a multigraph $G$ denoted ${\rm Aut}(G)$ is the set of permutations $\sigma \in {\rm Sym}(V)$ such that $M(\sigma(i),\sigma(j))=M(i,j)$ for all $i,j\in V$. \end{definition}
\begin{definition} A multigraph is transitive if the action of ${\rm Aut}(G)$ on $V(G)$ is. A multigraph is doubly transitive if for every vertex $v\in V$ and every $i,j\in V$ with $M(v,i)=M(v,j)>0$ there is an element $\sigma$ of the stabilizer of $v$ with $\sigma(i)=j$.
\end{definition}
\begin{definition} A cut on a multigraph $G$ is a partition of its vertex set $V(G)$ into two parts $\{S,S^{c}\}$. The weight of the cut $(S,S^{c})$ is 
\[w(S,S^{c})=\sum_{i<j: i\in S, j\in S^{c}} M(G)(i,j)\]
The maximum cut problem asks for the maximum weight among all cuts of $G$, that is, to determine
\[\MC(G):= \max_{\{S,S^{c}\}} w(S,S^{c})\]
\end{definition}
It is well known that the problem of determining whether the maximum cut of a graph is larger than a given value is NP-complete and thus it cannot be solved efficiently for all graphs unless ${\rm P}={\rm NP}$. A viable alternative is to use a polynomial time certified approximation algorithm such as the Goemans Williamson algorithm, described in the next section.

Finally we recall the following definition, which will be generalized in Section~\ref{srmgs}
\begin{definition} A graph $G$ is a strongly regular graph with parameters $(v,d,c,k)$ if it has $v$ vertices, it is regular of degree $d$, every two adjacent vertices have exactly $c$ common neighbors and every two disjoint vertices have exactly $k$ common neighbors. \end{definition}

\subsection{The Goemans-Williamson maxcut approximation algorithm and performance ratios.}\label{GWAlgo}

\subsubsection{A description of the algorithm}
Goemans and Williamson introduced in~\cite{GW} a certified stochastic approximation algorithm for the maximum cut problem. The algorithm proceeds in two stages: first, it introduces a semidefinite relaxation of the maximum cut problem (which can be solved in polynomial time) and then a random rounding procedure which allows us to produce cuts whose weight is guaranteed to be at least $\alpha\%\approx 87.8\%$ of the maximum cut of $G$. We describe these two steps in greater detail, 
\begin{enumerate}
\item{Semidefinite Relaxation: The maximum cut problem on a multigraph $G=(V,M)$ can be stated as a quadratic integer optimization problem. We assign one variable $x_i$ to each vertex and encode a cut $(S,S^{c})$ by letting $x_i=1$ if $i\in S$ and $x_i=-1$ otherwise. With this notation the maximum cut of $G$ equals
\[\MC(G)=\max_{ x_i\in \{-1,1\}} \sum_{i<j} M_{i,j}\frac{1-x_ix_j}{2}.\]
We can think of $x_i\in \{-1,1\}$ as an assignment from the vertices of $G$ to points in the $0$-sphere. More generally, for an assignment $f:V\rightarrow S^{p}$ we define
\[\SD(f):=\sum_{i<j} M_{i,j} \frac{1-f(i)\cdot f(j)}{2}\]
and letting $f$ run over all assignments of vertices to vectors in some sphere we have
\[\MC(G)\leq \max_{f} \SD(f) = \max_{X\succeq 0, X_{ii}=1}\sum_{i<j}M_{i,j} \frac{1-X_{ij}}{2}=:SD(G)\] 
Where the inequality occurs since the set of assignments includes the integral assignments $f(i)=x_ie_p$ and the equality because a symmetric matrix $X_{ij}$ is positive semidefinite iff it admits a Cholesky factorization. The determination of the rightmost quantity is a semidefinite optimization problem and thus can be solved in polynomial time~\cite{NN}. 
}
\item{ Randomized Rounding: For any assignment $f:V\rightarrow S^p$ and any hyperplane $H\in (\mathbb{R}^{p+1})^*$  we can produce a cut $(S(H),S(H)^c)$ by letting 
\[S(H)=\{i\in V: H(i)\geq 0\}.\] 
Goemans and Williamson study the weights of cuts produced by hyperplanes chosen uniformly at random in the dual unit sphere. They compute the expected value of the random weight $W(f)$ of cuts produced in this manner and relate it with the value of $\SD(f)$ (see~\cite{GW}[Theorems 2.1,2.3] for details),
\[ \mathbb{E}[W(f)]:=\mathbb{E}\left(w(S(H),S(H)^c)\right) = \sum_{i<j} \frac{\arccos(f(i)\cdot f(j))}{\pi} \geq \alpha \SD(f)\]
where
\[\alpha=\min_{0\leq\theta \leq \pi} \frac{2}{\pi}\frac{\theta }{1-\cos(\theta)}\approx 0.87856\]
If $f^*$ is an optimal embedding then $\MC(G)\leq SD(f^*)=SD(G)$ and thus the above randomized rounding procedure yields cuts whose weight is at least $87.85\%$ of $\MC(G)$.
}
\end{enumerate}
It is known~\cite{UGC} that if the unique games conjecture holds then the Goemans Williamson algorithm has the the best possible approximation ratio $\alpha$ for the maximum cut problem. 
\subsubsection{Performance ratios} \label{performanceRatios}
From the above analysis, we have the following chain of inequalities,
\[ \alpha SD(G)\leq \mathbb{E}[W(f)] \leq \MC(G)\leq SD(G)\] 
and thus the ratio $\frac{\mathbb{E}[W(f^*)]}{SD(G)}$ is a good one dimensional measure of the performance of the algorithm on a graph $G$. 
\begin{definition} The performance ratio of the GW algorithm on a graph $G$ is the quantity
\[\alpha_G:=\frac{\mathbb{E}[W(f^*)]}{SD(G)}\]
For a set of graphs $\mathcal{G}$ the performance ratio of the algorithm on $\mathcal{G}$ is the quantity
\[\alpha(\mathcal{G}):= \inf_{G\in \mathcal{G}} \alpha_{G}\]
and the essential performance ratio is given by
\[ e(\mathcal{G}) = \lim_{n\rightarrow \infty} \inf\{ \alpha_{G}: G\in \mathcal{G}, |G|\geq n\}\]
\end{definition}
Note that $\alpha\leq \alpha(\mathcal{G})\leq 1$ and that the quality of the semidefinite relaxation and the rounding technique on $\mathcal{G}$ are simultaneously controlled by $\alpha(\mathcal{G})$, improving as this quantity increases. On the other hand the essential performance ratio captures the behavior of the algorithm as we look at larger and larger instances. By a Theorem of Karloff~\cite{Karloff} it is known that $\alpha(\mathcal{G})=\alpha$ where $\mathcal{G}$ is the set of all transitive graphs.

\subsubsection{The semidefinite dual problem}
Let $G$ be a multigraph. For $\gamma:=(\gamma_1,\dots, \gamma_{|V|})\in \mathbb{R}^{|V|}$ let 
\[SD^*(\gamma):= \frac{1}{2}\sum_{i<j} M(i,j) +\frac{1}{4}\sum_{i\in V} \gamma_i\]
A simple direct calculation shows that the dual of the semidefinite relaxation of the maxcut problem for $G$ is
\[SD^{*}(G):=\min_{\gamma\in F}SD(\gamma)\text{   with  $F:=\{\gamma:  M+{\rm diag}(\gamma)\succeq 0\}$},\]
where ${\rm diag}(\gamma)$ is the diagonal matrix with ${\rm diag}(\gamma)_{ii}=\gamma_i$. Recall that by strong duality we have $SD(G)=SD^{*}(G)$ for every multigraph $G$.

\section{The geometry of the varieties $X_{a,b,c}$.}\label{geometry}
In this section we recall some basic facts about the geometry of varieties $X_{a,b,c}$ studied by Mukai~\cite{Mukai} and Castravet-Tevelev~\cite{CT}. Let $a,b,c$ be positive integers with $a,c\geq 2$ and let $r:=b+c$. Assume $a\leq c$ always and $c>2$ if $a=2$. Let $T_{a,b,c}$ be a $T$ shaped tree with $a+b+c-2$ vertices and let $X_{a,b,c}:={\rm Bl}_{b+c}\left( (\mathbb{P}^{c-1})^{a-1}\right)$ be any variety obtained by blowing up $b+c$ sufficiently general points in the product. 

Henceforth we assume $T_{a,b,c}$ is the Dynkin diagram of a finite root system  (i.e. $\frac{1}{a}+\frac{1}{b}+\frac{1}{c}>1$ ). These varieties can be thought of as higher-dimensional analogues of del Pezzo surfaces (obtained when $a=2$, $c=3$ and $r\leq 8$) and share many of their fundamental properties. To describe the similarities we need to introduce the following terminology, 

\begin{definition} The Picard group ${\rm Pic}(X_{a,b,c})$ is the free $\mathbb{Z}$-module of rank $a+r-1$ generated by  the classes $H_1,\dots, H_{a-1}$ of pullbacks of the hyperplane sections of the factors $\mathbb{P}^{c-1}$ together with $r$ classes of the exceptional divisors above the blown up points $E_1,\dots, E_r$. The canonical class is 
\[K:=-c\sum_{i=1}^{a-1}H_i + \left(ac-a-c\right)\sum_{j=1}^rE_j.\] 

Define a symmetric bilinear form on ${\rm Pic}(X_{a,b,c})$ by 
\[
\begin{array}{lll}
E_i\cdot E_j = -\delta_{ij} & H_i\cdot H_j=c-1-\delta_{ij} & H_i\cdot E_j=0\\
\end{array}
\]
\end{definition}
The following Lemma~\cite[Lemma 2.1]{CT} clarifies the relationship between the combinatorics of $T_{a,b,c}$ and the geometry of $X_{a,b,c}$.  
\begin{lemma} ${\rm Pic}(X_{a,b,c})$ has another basis $\alpha_1,\dots, \alpha_{a+r-2}, E_r$ where
\[\alpha_1=E_1-E_2, \dots, \alpha_{r-1} = E_{r-1}-E_r \]
\[\alpha_{r}=H_1-E_1-\dots -E_c\]
\[\alpha_{r+1}=H_2-H_1,\dots, \alpha_{r+a-2}= H_{a-1}-H_{a-2}\]
Moreover, $\alpha_1,\dots, \alpha_{r+a-2}$ are a basis for the orthogonal complement $K^{\perp}$ and a system of simple roots of a finite root system with Dynkin diagram $T_{a,b,c}$ (see Figure~\ref{Tabc}).
\end{lemma}
As a result, there is an action of the Weyl group of the root system which extends to ${\rm Pic}(X)$ by fixing the canonical divisor. Moreover the above product is invariant under this action. The action allows us to define a collection of distinguished classes analogous to the $(-1)$-curves on del Pezzo surfaces.
\begin{definition} We say that a class $V\in \Pic(X_{a,b,c})$ is a $(-1)$-divisor if it belongs to the orbit of $E_r$ under the action of the Weyl group. 
\end{definition}
By results of Dolgachev~\cite{Dolga} each of these classes is exceptional in some small modification of $X_{a,b,c}$ and in particular every such class determines a distinguished divisor. As in the case of del Pezzo surfaces it is sometimes possible to set up a correspondence between the $(-1)$-divisors and the weights of an irreducible representation of the semisimple lie algebra with Dynkin diagram $T_{a,b,c}$.  More precisely, 
\begin{definition} Let $\mathfrak{g}_{a,b,c}$ be a semisimple Lie algebra with Dynkin diagram $T_{a,b,c}$. Let $\Lambda \subseteq K^{\perp}\otimes \mathbb{Q}$ be the weight lattice spanned by the fundamental weights $\omega_1,\dots, \omega_{a+r-2}$ defined by $\omega_i\cdot \alpha_j=\delta_{ij}$. For each $\omega\in \Lambda$ let $L_{\omega}$ be the irreducible representation with highest weight $\omega$. The representation $L_{\omega}$ is called minuscule if its weights are precisely the elements of the orbit of $\omega$ under the Weyl group action.
\end{definition}
Since $E_r\cdot \alpha_j = \delta_{j,r-1}$ the orthogonal projection of $E_r$ onto $K^{\perp}$ is $\omega_{r-1}$ and thus the orthogonal projection determines a natural bijection between the $(-1)$-divisors and the weights of an irreducible representation precisely when $L_{\omega_{r-1}}$ is a minuscule representation. The classification of minuscule representations, or equivalently, minuscule Dynkin diagrams $T_{a,b,c}$ is well known (see Figure~\ref{DD}). The only arising cases are:
\[
\begin{array}{ll}
X_{s+1,1,n+1} & A_{s+r-1}\\
X_{2,2,n+1} & D_r\\
X_{2,3,3} & E_6\\
X_{2,4,3} & E_7\\
\end{array}
\]
Where the last two rows correspond to Del Pezzo surfaces of degrees three and two respectively. 

\begin{definition} We define the multigraph of exceptional divisors $G_{a,b,c}$ to be the multigraph whose vertices are the $(-1)$ divisors and with weight $M_{a,b,c}(U,V)=U\cdot V$
for distinct vertices $U$ and $V$. The maximum cut problem for the varieties $X_{a,b,c}$ asks for determining the maximum cut of the multigraphs $G_{a,b,c}$. Equivalently: Among all partitions of the classes of $(-1)$-divisors on $X_{a,b,c}$ into two sets, what is the largest possible number of pairwise intersections? 
\end{definition}
We address the maximum cut problem for minuscule varieties $X_{a,b,c}$ by showing that the normalized orthogonal projection  $f: V(G_{a,b,c})\rightarrow K^{\perp}$ is optimal for the Goemans-Williamson semidefinite relaxation of maxcut. The above bijections between $(-1)$-divisors and weights of minuscule representations will be a key ingredient of the proof.

\section{The Goemans-Williamson algorithm on symmetric multigraphs.}\label{GWsymmetry}

In this section we study the behavior of the Goemans-Williamson algorithm on graphs with symmetries. The main idea is that under a sufficiently transitive group action, optimal solutions of the optimization problems under consideration can always be chosen to be invariant, allowing us to reduce the computation of $\alpha_G$ to the question of determining the smallest eigenvalue of $M(G)$. 

This point of view allows us to study the behavior of the Goemans-Williamson algorithm on some classes of strongly regular graphs and to show that, in marked contrast with worst case performance results the performance of the GW algorithm becomes optimal as the number of vertices increases (See Section~ref{srmgs} for details). This result is especially interesting in the light of recent results~\cite{UGC} showing that the existence of an MaxCut polynomial time algorithm with an approximation ratio greater than $\alpha$ would imply the falsehood of the unique games conjecture. Classes of graphs where the performance of the GW algorithm is provably greater than $\alpha$ are thus one natural place to look for counterexamples. Moreover, our results imply that the semidefinite relaxation is an asymptotically accurate formula for the value of the maximum cut problem on such graphs. Finally, in this section we also introduce the concept of strongly regular multigraph which generalizes the idea of strongly regular graph and characterize their spectra. The class of strongly regular multigraphs is the natural context to study the graphs of $(-1)$-divisors $G_{a,b,c}$. 

\begin{lemma} \label{symAlpha} Let $G$ be a transitive multigraph and let $\lambda_1(G)$ be the smallest element of the spectrum of $M$. The following statements hold:
\begin{enumerate}
\item{\label{symmetric} $SD^{*}(G) = \frac{1}{2}\sum_{i<j} M(i,j) -\frac{|V|\lambda_1(G)}{4}$.}
\item{If $G$ is a doubly transitive graph then 
\begin{enumerate}
\item{ \label{opt2Transitive}There exists an embedding $f$ such that the angle between every two adjacent vertices is a constant $\eta$ satisfying $\cos(\eta) =\frac{\lambda_1(G)}{d}$ where $d$ is the degree of $G$.}
\item{ \label{ratio2Transitive} The performance ratio for these graphs equals
\[\alpha_G:=\frac{2}{\pi} \frac{\arccos\left(\frac{\lambda_1}{d}\right)}{1-\frac{\lambda_1}{d}}\]
 }
\end{enumerate}
}
\end{enumerate}
\end{lemma}
\begin{proof} The group $H:={\rm Aut}(G)$ acts on $\mathbb{R}^{|V|}$ by permutation of its components and on $|V|\times |V|$ matrices by simultaneously permuting rows and columns. Let $\gamma=(\gamma_1,\dots, \gamma_n)$ be an optimal solution of the dual of the semidefinite relaxation and define $\overline{\gamma}:=\frac{1}{|H|}\sum_{g\in H} g\cdot\gamma$. The point $\overline{\gamma}$ is feasible since 
\[0\preceq \frac{1}{|H|}\sum_{g\in H}g\cdot(M+{\rm diag}(\gamma))=M+{\rm diag}(\overline{\gamma})\] 
and also optimal since 
\[\frac{1}{|H|}\sum_{g\in H}\left(SD^*(\gamma)\right) = SD^*\left(\frac{1}{|H|}\sum_{g\in H}g\cdot \gamma \right)=SD^*(\overline{\gamma}).\]
By transitivity of the action of $H$ the components of $\overline{\gamma}$ are identical with constant value $c$. Since $M+cI\succeq 0$ it follows that $c+\lambda_1\geq 0$. 
As a result 
\[SD^{*}(-\lambda_1(1,\dots, 1))\leq SD^{*}(c(1,\dots,1))\] 
and by optimality $c=-\lambda_1$. Evaluating the right hand side we obtain claim~(\ref{symmetric}.)
If the action of $H$ is doubly transitive let $X$ be an optimal solution for the semidefinite relaxation of maxcut and note that, as before the average $\overline{X}:=\frac{1}{|H|}\sum_{g\in H}g\cdot X$ is an optimal feasible solution. Since the action is doubly transitive $\overline{X}_{ij}$ has only two possible values depending on whether vertices $i$ and $j$ intersect and in particular computing the objective function of the semidefinite relaxation we have
\[SD(G)=SD(\overline{X})= \frac{e}{2}(1-\cos(\eta))\]
where $e$ is the number of edges of $G$. Since there is no duality gap, by part~(\ref{symmetric}.) we have
\[ \frac{e}{2}(1-\cos(\eta)) = \frac{e}{2}-\frac{|V|\lambda_1}{4}=\frac{e}{2}\left(1-\frac{\lambda_1}{d}\right)\] 
where the last equality follows from $dv=2e$ since $2$-transitive graphs are regular. This establishes claim~(\ref{opt2Transitive}.). 
Let $f$ be an embedding obtained from the Cholesky factorization of $\overline{X}$. The expected weight of a random hyperplane cut obtained from $f$ is
\[\mathbb{E}[W(f)]=\frac{1}{\pi}\sum_{i<j} M(i,j)\arccos(\overline{X}_{ij})=\frac{e\arccos(\frac{\lambda_1}{d})}{\pi}\]
and claim~(\ref{ratio2Transitive}.) follows. 
\end{proof}
\begin{remark}For nonnegative integers $m\geq t\geq q$, let $J(m,t,q)$ be a graph whose vertices are the sets $\binom{m}{t}$, two of them adjacent iff they intersect in a set of size $q$. It is easy to see that the action of the permutation group on these graphs is doubly transitive. The spectra of this graphs was computed by Knuth~\cite{Knuth} and the above result simplifies the proof of a Theorem of Karloff~\cite{Karloff} showing that, for the set $\mathcal{J}$ of all graphs $J(m,t,q)$ we have the equality $\alpha(\mathcal{J})=\alpha$. In particular, even for doubly transitive graphs the worst-case performance of the GW algorithm is equal to its theoretical lower bound.
\end{remark}
\begin{remark} If the automorphism group $H$ of a graph $G$ has rank three (i.e. if there are exactly three orbits of $H$ on $V(G)\times V(G)$, namely equal, adjacent and nonadajcent pairs) then $G$ is two-transitive and in fact $G$ is a strongly regular graph so in particular its spectrum has only three values simplifying the computation of the performance ratio (see Theorem~\ref{multigraphSpectra} for details). 
\end{remark}

\subsection{Strongly regular multigraphs}\label{srmgs}
In this section we introduce the concept of strongly regular multigraph which is the natural context for our analysis of the multigraphs $G_{a,b,c}$.\\
\begin{definition} A multigraph $(V,M)$ is a strongly regular multigraph if it satisfies the following two conditions:
\begin{enumerate}
\item{There exist real numbers $a,b$ such that $M^2_{ij} = a M_{ij} +b$ for every $i\neq j$.}
\item{There exists real numbers $c$ and $d$ such that $MJ = d J$ and $M^2_{ii}=c$ where $J$ is the all-ones matrix.}
\end{enumerate}
\end{definition}
If the entries of $M$ are integral and we think of $M(i,j)$ as the number of paths between vertices $i$ and $j$ the above conditions say that every vertex has equal degree $d$, that the number of paths of length two between two distinct vertices is an affine linear function of the number of paths between them and that the number of length two loops starting at any vertex is independent of the vertex. 

It is immediate from the definition that every strongly regular graph is a strongly regular multigraph. Moreover, the spectral properties of strongly regular multigraphs are very similar to those of strongly regular graphs (see for instance~\cite{VW}[Chapter 21]), 

\begin{theorem} \label{multigraphSpectra} The spectrum of a strongly regular multigraph $(V,M)$ has two possibilities. Either
\begin{enumerate}
\item{ Consists of exactly two values $\frac{-d}{n-1}<d$ with multiplicities $n-1$ and $1$ respectively or }
\item{ Consists of exactly three values $\eta_-<\eta_+<d$ with corresponding multiplicities $f_-,f_+,1$ given by
\[
\begin{array}{ccc}
\eta_{\mp} = \frac{a\mp \sqrt{a^2-4(c-b)}}{2} & &f_{\mp} = \pm \frac{ d + (n-1)\eta_\pm}{\eta_+-\eta_-}\\
\end{array}
\] 
}
\end{enumerate}
\end{theorem}
\begin{proof} From the hypotheses we see that the adjacency matrix satisfies the equations
\[
\begin{array}{ccc}
MJ=dJ & & M^2=aM + bJ +(c-b)I\\
\end{array}\]
where $J$ is the all ones matrix. By the Perron-Frobenius theorem the vector of ones is an eigenvector with eigenvalue $d$ and multiplicity one. By symmetry of $M$ the remaining eigenvectors $v$ are orthogonal to $J$ and thus their eigenvalues $\eta$ satisfy the quadratic equation
\[\eta^2 v = a\eta v +(c-b)v\]
so either $a^2=4(b-c)$ and there are only two eigenvalues $\frac{a}{2}<d$ with multiplicities $n-1$ and $1$  satisfying $(n-1)\frac{a}{2}+d=tr(M)=0$ or the quadratic equation has two distinct roots $\eta_{-}$ and $\eta_+$ as above. Moreover, if $f_-$ and $f_+$ denote their respective multiplicities we have
$1+f_-+f_+=n$ and $f_-\eta_{-}+f_+\eta_++d={\rm tr}(M)=0$. Solving these equations we get the above result.
\end{proof}
As we will show in Theorem~\ref{main}, the minuscule graphs $G_{a,b,c}$ are strongly regular multigraphs. 
\subsection{Strongly regular graphs with spectrum bounded below}
Let $m$ be a nonnegative integer and define $\mathcal{R}(-m)$ to be the collection of transitive strongly regular graphs with smallest eigenvalue $-m$. The following Theorem shows that as the number of vertices increases the performance ratio of the Goemans-Williamson algorithm approaches optimum.

\begin{theorem}  The essential performance ratio $e(\mathcal{R}(-m))$ equals one, where
\[ e(\mathcal{R}(-m)) = \lim_{n\rightarrow \infty} \inf\{ \alpha_{G}: G\in \mathcal{R}(-m), |G|\geq n\}\]
\end{theorem}
\begin{proof} By Neumaier's classification~\cite{Neumaier} it is known that a strongly regular graph with smallest eigenvalue $-m$ is one of either,
\begin{enumerate}
\item{A finite list of exceptional graphs $L(m)$.}
\item{A complete multipartite graph with parts of size $m$.}
\item{A Steiner triple system with blocks of size $m$.}
\item{The graph of $(m-2)$ mutually orthogonal $n\times n$ latin squares (whose vertices are the $n^2$ squares and two squares are adjacent if either lie on the same row or column or have the same symbol in two of the latin squares).}
\end{enumerate}
In any of the last three classes the degree of the graph increases as the number of vertices does. The result then follows from part $(2)$ of Lemma~\ref{symAlpha}.
\end{proof}
\section{The maximum cut problem on minuscule $X_{a,b,c}$ varieties.}\label{maxcutMain}
In this section we show that the normalized orthogonal projection of the classes of $(-1)$-divisors to $K^{\perp}$ is optimal for the semidefinite relaxation of the maximum cut problem. As a result we obtain asymptotically sharp bounds for the maximum cut problem on $X_{a,b,c}$. In Section~\ref{Simulations} we use the optimal embeddings and stochastic simulation to improve the inequalities for small values of the parameters and determine the exact value of the maximum cut in some cases. These results give another instance of ``optimality" associated to root systems of type $A$,$D$,$E$. 

Let $T_{a,b,c}$ be the Dynkin diagram of a finite root system. Let $r:=b_c$ and $\delta:=bc+ac+ab-abc$. A simple calculation shows that $K^2= (ac-a-c)\delta$ so the orthogonal projection $q:{\rm Pic}(X_{a,b,c})\rightarrow K^{\perp}$ of a $(-1)$-divisor is given by
\[ q(V) = V + \frac{1}{\delta} K \]  
as a result,

\begin{definition} The normalized orthogonal projection $f:V(G_{a,b,c})\rightarrow K^{\perp}$ is given by
\[ f(V):= \frac{1}{\sqrt{1+\frac{ac-a-c}{\delta}}}  \left(V+\frac{1}{\delta}K\right)\]
\end{definition}

\begin{lemma}\label{evaluation} For each integer $k$ let $S_k(a,b,c)$ be the number of edges of $G_{a,b,c}$ of weight $k$. The following statements hold:
\begin{enumerate}
\item{\label{SDabc} \[SD(f)= \frac{\delta}{\delta+(ac-a-c)} \sum_{k\in \mathbb{N}} S_k(a,b,c) \frac{k(k+1)}{2}.\]}
\item{\label{Eabc} The expected weight of a cut obtained from $f$ via a uniformly distributed random hyperplane is
\[ \mathbb{E}[W(f)]= \sum_{k\in \mathbb{N}} S_k(a,b,c)k\arccos\left(1-\frac{\delta}{\delta +ac-a-c}(1+k)\right).\]
}
\item{\label{SDDabc} Let $\lambda_1(a,b,c)$ be the smallest eigenvalue of $M_{a,b,c}$. Then $\gamma:=-\lambda_1(a,b,c)(1,\dots,1)$ is a feasible solution of the semidefinite dual problem and
\[SD^*(\gamma)= \frac{1}{2}\left(\sum_{k\in \mathbb{N}} kS_k(a,b,c)\right)-\frac{|V|\lambda_1(a,b,c)}{4}.\]
}
\end{enumerate}
\end{lemma}
\begin{proof} Recall that the intersection form is negative definite in $K^{\perp}$ and thus its negative defines an inner product  $\langle, \rangle$ making $K^{\perp}$ into an euclidean space. In particular, for any two $(-1)$-divisors $U$,$V$ we have
\[1-\langle f(U),f(V)\rangle = \frac{\delta}{\delta+(ac-a-c)}(1+U\cdot V)\]
and thus
\[SD(f):= \frac{1}{2}\sum_{i<j} M_{ij} (1-\langle f(i),f(j)\rangle)= \sum_{k\in \mathbb{N}}\sum_{i<j:M_{ij}=k} k(1-\langle f(e_0),f(e_1)\rangle) \]
proving claim~(\ref{SDabc}.). (\ref{Eabc}.) By ~\cite{GW}[Theorems 2.1,2.3]  the expected weight $\mathbb{E}[W(f)]$ of a cut obtained from $f$ by a random hyperplane is
\[\sum_{i<j} \frac{\arccos\langle f(i), f(j)\rangle}{\pi}= \sum_{k\in \mathbb{N}} \sum_{i<j:M_{ij}=k} k\arccos\left(1-\frac{\delta}{\delta +ac-a-c}(1+k)\right)\]
as claimed. (\ref{SDDabc}.) Since $M_{a,b,c} -\lambda_1I\succeq 0$ the vector $\gamma$ as defined above is a feasible point of the dual problem. By a straightforward calculation the above formula gives the value of the objective function of the dual problem at $\gamma$. 
\end{proof}

The following Lemma gives a geometric interpretation of the normalized orthogonal projection $f$ as a placement of $(-1)$-divisors on the vertices of certain Coxeter matroid polytopes. Note that the edges of the multigraph are not on the boundary of the polytope.

\begin{lemma}\label{bij}  Let $T_{a,b,c}$ be a minuscule Dynkin diagram. The images under $f$ of the classes of $(-1)$ divisors in $K^{\perp}$ are precisely the vertices of the Coxeter polytopes below. Moreover these classes are in canonical bijection with the weights of the representations in the last column.
\[
\begin{array}{l|l|l|l|l}
\text{Tree} & \text{Type} & \text{Coxeter polytope} & \text{Lie Algebra} & \text{Representation}\\ 
\hline
T_{s+1,1,n+1} & A_{r+s-1} & \text{Hypersimplex $\Delta(r-1,r+s)$} &  \mathfrak{sl}_{r+s} & \bigwedge^{r-1}(V^{r+s})\\
T_{2,2,n} & D_r & \text{Demicube $\Phi(r)$} & \mathfrak{so}_{2r} & \text{Spin rep $S^+$}\\
T_{2,3,3} & E_6 & \text{Gossett polytope $2_{21}$} & \mathfrak{e}_6 & J\\
T_{2,4,3} & E_7 & \text{Hess Polytope $3_{21}$} & \mathfrak{e}_7 & W\\  
\end{array}
\]
\end{lemma}
\begin{proof} Given a root system $R\subseteq V$ and a point $p$ not belonging to all hyperplanes orthogonal to the roots in $R$ the generalized permutahedron defined by $p$ is the convex hull of the orbit of p under the action of the Weyl group of the root system (see~\cite{BGW} for details). The stabilizer of $p$ under the Weyl group is completely determined by its dot product with the roots in a simple system (see ~\cite{BGW}[Lemma 6.2.1]). Any two points with maximal stabilizers which are orthogonal to the same set of simple roots must differ only by length and thus must lead to isomorphic generalized permutahedra. The above bijections follow from this observation.  Using the notation from Section~\ref{geometry} the class $f(E_r)$ is orthogonal to all roots $\alpha_i$ for $i\neq r-1$ and thus the convex hull of its orbit is the generalized permutahedron whose Coxeter diagram has a ring around $\alpha_{r-1}$ as in Figure~\ref{DD}. These polytopes are well known and correspond to those in the table.
Similarly, let $\omega_{r-1}$ be the irreducible representation of $\mathfrak{g}_{a,b,c}$ with weight dual to $\alpha_{r-1}$. In the cases above this representation is minuscule and thus its weights are a single orbit under the action of the Weyl group establishing the desired bijection. The representations in the last three rows of the above table are the half-spin representation $S^{+}$ of $\mathfrak{so}_{2r}$ (see~\cite{FH}[Lecture 20]), the $27$-dimensional representation $J$ of the Lie algebra $\mathfrak{e}_6$ corresponding to infinitesimal norm similarities of the exceptional Jordan algebra and $W$ is the fundamental $56$-dimensional representation of the Lie algebra $\mathfrak{e}_7$ (see~\cite{LRT}[Chapter 6] for descriptions).
\end{proof}
The main result of this article is the following,
\begin{theorem}\label{main} Let $T_{a,b,c}$ be a minuscule Dynkin diagram,
\begin{enumerate} 
\item{ \label{spectrumTabc} The multigraph $G_{a,b,c}$ is a strongly regular multigraph. Its spectrum has three values $1+\lambda < 1 < 1+\eta$ where
\[ 
\begin{array}{l|l|l|l}
Graph & Type & \lambda & \eta\\
\hline
G_{s+1,1,n+1} & A_{r+s} & -\binom{r+s-2}{r-2} & s^2\binom{r+s-2}{r-3} -s\binom{r+s-2}{r-2} +\binom{r+s-2}{r-1} \\
G_{2,2,n} & D_r & -2^{r-3} & (r-4)2^{r-3}\\
G_{2,3,3} & E_6 & -6 & 9\\
G_{2,4,3} & E_7 & -12 & 28\\  
\end{array}
\]
}
\item{ \label{optimalityTabc} The normalized orthogonal projection $f:V(G_{a,b,c})\rightarrow K^{\perp}$ is optimal for the geometric relaxation of the maximum cut problem.}
\item{ \label{boundsTabc} Let $m(a,b,c)$ denote the value of the maximum cut problem for $X_{a,b,c}$.  We have $\lceil\ell(a,b,c)\rceil \leq m(a,b,c)\leq \lfloor u(a,b,c)\rfloor$ where 
\begin{tiny}
\[ 
\begin{array}{l|l|l}
Graph & \ell(a,b,c) & u(a,b,c)\\
\hline
G_{s+1,1,n+1} & \frac{1}{2\pi}\binom{r+s}{r-1}\sum_{k=0}^s \binom{s+1}{k+1}\binom{r-1}{k+1}k\arccos\left( 1-\frac{(r+s)(k+1)}{(s+1)(r-1)}\right) & \frac{r+s}{2(s+1)(r-1)} \binom{r+s}{r-1}\binom{s+1}{2}\binom{r+s-2}{r-3}\\
G_{2,2,n} & \frac{2^{r-2}}{\pi}\sum_{k=0}^{\lfloor\frac{r}{2}\rfloor}k \binom{r}{2(k+1)}k\arccos\left(1-\frac{4(k+1)}{r}\right)  & (r-3)2^{2r-6}\\
G_{2,3,3} & 90 & 101.25\\
G_{2,4,3} & 516 & 560\\  
\end{array}
\]
\end{tiny}
Moreover $\lim_{|V|\rightarrow \infty} \frac{m(a,b,c)}{u(a,b,c)}=1$ so the percentage error of approximating the maxcut by its upper bound is asymptotically zero on the infinite families.
}
\item{The upper bound agrees with the maximum cut for $G_{2,2,n}$ with $n\leq 8$ and for $G_{s+1,1,n+1}$ with $r-s=2$, $r\leq 7$.}
\end{enumerate}
\end{theorem}
The proof of the above Theorem appears in Section~\ref{proofofMainTheorem} after the different types have been analyzed independently in the next three sections. A key element is provided by the following purely combinatorial interpretations of the multigraphs $G_{a,b,c}$ for infinite minuscule families. 

\begin{definition} Let $C_{s+1,1,n+1}$ be the multigraph whose vertices are the $r-1$-subsets of $[s+r]$ with weight $M(T,S):=|T^{c}\cap S|-1$ between any two distinct vertices $S$ and $T$. Let $C_{2,2,n}$ be the multigraph whose vertices are the even subsets of $[r]$ with weight $M(T,S):=\frac{|T\ast S|}{2}-1$ between any two distinct vertices $S$ and $T$. Here $S\ast T:=(S\cup T)\setminus (S\cap T)$ is the symmetric difference between $S$ and $T$.\end{definition}

\begin{lemma} The bijection between exceptional divisors and weights of fundamental representations gives multigraph isomorphisms $G_{s+1,1,n+1}\cong C_{s+1,1,n+1}$ and $G_{2,2,n}\cong C_{2,2,n}$.
\end{lemma}
\begin{proof} We look at the bijections from Lemma~\ref{bij} in more detail. We use them to obtain combinatorial interpretations of the product $U\cdot V$ of divisors.
Let $\mathfrak{h}\subseteq \mathfrak{sl}_{r+s}$ be the Cartan subalgebra of traceless matrices and recall that the quotient $\mathfrak{h}^*:=\langle L_1,\dots, L_{r+s}\rangle /(\sum_{i=1}^{r+s} L_i)$ can be split realizing $\mathfrak{h}^{*}$ as the subspace of $\langle L_1,\dots L_{r+s}\rangle$ spanned by the simple roots $\beta_i:=L_i-L_{i+1}$ for $1\leq i\leq r+s-1$. The weights of the fundamental representation $\bigwedge^{r-1}(V^{r+s})$ are precisely $v_J:=\bigwedge_{i\in J} e_i$ for $J\in \binom{r+s}{r-1}$ and the highest weight vector is $v_{1,\dots, r-1}$ dual to $\beta_{r-1}$. Using the splitting, the normalized weight of the vector $v_J$ is $\omega_J\in \mathfrak{h}^{*}$ given by
\[ \omega_J:=\frac{1}{\sqrt{(r+s)(s+1)(r-1)}} \left( (s+1)\sum_{j\in J} L_j +(1-r)\sum_{j\in J^{c}} L_j\right)\] 
By a direct computation we have $\omega_J\cdot \omega_T = |J\setminus T| \frac{s+r}{(r-1)(s+1)}-1$.
On the other hand, if $V_J$ and $V_T$ are the $(-1)$-divisors corresponding to $J$ and $T$ under the bijection in Lemma~\ref{bij} then
\[ f(V_J)\cdot f(V_T) = \frac{r+s}{(s+1)(r-1)}\left( V_J\cdot V_T -\left(\frac{(s+1)(r-1)}{r+s}-1\right)\right)\]
Since the bijection is an isometry it follows that $V_J\cdot V_T= |J\setminus T|-1$. As a result $G_{s+1,1,n+1}\cong C_{s+1,1,n+1}$.
For the $G_{2,2,n}$ recall (~\cite{FH}[Lecture 18]) that the even orthogonal Lie algebra $\mathfrak{so}_{2r}$ can be realized as $2r\times 2r$ matrices with four $r\times r$ blocks where the diagonal blocks are negative transposes of each other and the off diagonal blocks are skew-symmetric. The Cartan subalgebra $\mathfrak{h}\subseteq \mathfrak{so}_{2r}$ corresponds to the diagonal matrices satisfying these restrictions and in particular the operators $L_i\in \mathfrak{h}^*$ which extract the $i$-th diagonal entry  
for $1\leq i\leq r$ are a basis for $\mathfrak{h}^*$. A system of simple roots consists of $\gamma_i:=L_i-L_{i+1}$ for $1\leq i\leq r-1$ and $\gamma_r:=L_{r-1}+L_{r}$.
The fundamental representation with weight dual to $\gamma_{r}$ is precisely the even spin representation with highest weight $\frac{1}{2}(L_1+\dots +L_r)$ (see~\cite{FH}[Lecture 20] for a precise description). Suffices to say (\cite{FH}[Proposition 20.15]) that the underlying vector space is $\bigwedge^{even}(W)$ where $W=\langle e_1,\dots, e_r\rangle$. The  
weight vectors of this representation are $e_J:=\bigwedge_{j\in J}e_j$ with normalized weights $\omega_J:=\frac{1}{\sqrt{r}}\left(\sum_{j\in J} e_j-\sum_{j\not\in J} e_j\right)$ as $J$ runs over the even subsets of $[r]$. 
A direct computation yields $\omega_J\omega_T=\frac{2}{r}|J\ast T|-1$ where $J\ast T$ denotes the symmetric difference.  On the other hand, if $V_J$ and $V_T$ ae the $(-1)$-divisors corresponding to $\omega_J$ and $\omega_T$ we have $f(V_J)\cdot f(V_T)=\frac{1}{r}\left(4V_J\cdot V_T-(r-4)\right)$. Since the bijection is an isometry it follows that $V_J\cdot V_T=\frac{|J\ast T|}{2}-1$. As a result $G_{2,2,n}\cong C_{2,2,n}$ as claimed.
\end{proof}

\subsection{Type $A$ multigraphs.}

Let $M$ be the adjacency matrix of the multigraph $C_{s+1,1,n+1}$. Recall that $r=n+2$ and define $B:=M-I$. We will show that the matrix $B$ satisfies a certain quadratic equation allowing us to conclude that the graphs $C_{s+1,1,n+1}$ are strongly regular multigraphs and to determine their spectra. To this end we will use the following elementary combinatorial identity, which holds for every two nonnegative integers $s$ and $m$,
\[\sum_{j=0}^{s} j\binom{s}{j}\binom{m}{m-j} = s\binom{m+s-1}{m-1} \]
\begin{lemma}\label{spectrumArps} The matrix $B$ satisfies the equation $B^2 = \lambda_{s,r}B + \eta_{s,r}J$ with
\[
\begin{array}{ll}
\lambda_{s,r}:=-\binom{r+s-2}{r-2} & \eta_{s,r}:= s^2\binom{r+s-2}{r-3}-s\binom{r+s-2}{r-2}+\binom{r+s-2}{r-1}\\ 
\end{array}
\]
As a result $C_{s+1,1,n+1}$ as a strongly regular multigraph and its spectrum consists of  
\[1+\lambda_{s,r}< 1< 1 + \eta_{s,r}\] 
with multiplicities $r+s-1$, $\binom{r+s}{r-1}-(r+s)$ and one respectively. 
\end{lemma}
\begin{proof} Since the action of the Weyl group is transitive on $C_{s+1,1,n+1}$ it suffices to show that the following equality holds for any set $T\in \binom{r+s}{r-1}$
\[ B^2_{[r-1],T}=\lambda_{s,r}B_{[r-1],T}+\eta_{s,r}\]
For $S\in \binom{r+s}{r-1}$ let $S_u:=S\cap [r+1,\dots, r+s]$, $S_d:=S\setminus S_u$ and let $s_u=|S_u|$ and $s_l:=|S_l|$. 
with this notation we have
\[B^2_{[r-1],T}=\sum_{S\in \binom{r+s}{r-1}} B_{[r-1],S}B_{S,T}= \sum_{S\in \binom{r+s}{r-1}}(s_u-1)(|S\cap T^c|-1)= \sum_{j=0}^{s+1}(j-1)\sum_{S\in \binom{r+s}{r-1}, s_u=j}|S\cap T^c|-1=\] 
\[=-\sum_{j=0}^{s+1} (j-1)|\{S:s_u=j\}| +\sum_{j=0}^{s+1}(j-1)\sum_{t\in T^c}|\{S:s_u=j, S\ni t\}|\]
the last sum can be divided into $t\in T^c\cap [r,r+s]$ and $t\in T^c\cap [r-1]$. We thus have,
\[-\sum_{j=0}^{s+1}(j-1)\binom{s+1}{j}\binom{r-1}{r-1-j}+ (s+1-t_u) \sum_{j=0}^{s+1} (j-1)\binom{s}{j-1}\binom{r-1}{r-1-j} +\] 
\[+ (r-1-t_d)\sum_{j=0}^{s+1}(j-1)\binom{s+1}{j}\binom{r-2}{r-2-j}\]
which gives the claimed equality using the combinatorial identity above and the equalities $s+1-t_u=s-B_{[r-1],T}$ and $r-1-t_d=1+B_{[r-1],T}$. Arguing as in the proof of Theorem~\ref{multigraphSpectra} we obtain, from the quadratic equation, the above description of the spectrum of $C_{s+1,1,n+1}$.
\end{proof}

The following combinatorial Lemma will be needed for the proof of our main theorem for the graphs $C_{s+1,1,n+1}$. Recall that $S_k:= S_k(s+1,1,n+2)$ is the number of edges of weight $k$ in the graph $C_{s+1,1,n+1}$

\begin{lemma} \label{combinatAr} The following equalities hold for $k\geq 0$,
\[
\begin{array}{l}
S_k = \frac{1}{2}\binom{r+s}{r-1}\binom{s+1}{k+1}\binom{r-1}{r-1-(k+1)}\\
\sum_{k=0}^s S_k = \frac{1}{2}\binom{r+s}{r-1}\left(\binom{r+s}{r-1}-1\right)\\
\sum_{k=0}^s S_k(k+1) = \frac{1}{2}\binom{r+s}{r-1} (s+1)\binom{r+s-1}{r-2}\\
\sum_{k=0}^s S_k\binom{k+1}{2}=\frac{1}{2}\binom{r+s}{r-1} \binom{s+1}{2} \binom{r+s-2}{r-3}\\
\end{array}
\]
\end{lemma}
\begin{proof} Since the multigraph $C_{s+1,1,n+1}$ is transitive, the number of edges of weight $k$ coming out of every vertex is the same. for the vertex $[r-1]$ this quantity is precisely the cardinality of the sets $S\in \binom{s+r}{r-1}$ with $|S\cap [r,r+s]|-1=k$ which equals $\binom{s+1}{k+1}\binom{r-1}{r-1-(k+1)}$. Summing over all vertices we see that 
\[2S_k = \binom{r+s}{r-1}\binom{s+1}{k+1}\binom{r-1}{r-1-(k+1)}\]
and the first formula follows. Now, using the formula for $S_k$ the last quantity becomes
\[\sum_{k=0}^s\binom{k+1}{2}S_k=\frac{1}{2}\binom{r+s}{r-1}\sum_{k=0}^s \binom{k+1}{2}\binom{s+1}{k+1}\binom{r-1}{r-1-(k+1)}.\]
The rightmost sum can be interpreted as the counting the number of ways of choosing a team of $r-1$ active players out of $r+s$ possible players by first choosing two captains from among $s+1$ distinguished eligible players and then choosing $r-3$ among the remaining $r+s-2$ players. As a result we have
\[ \sum_{k=0}^s\binom{k+1}{2}S_k = \frac{1}{2}\binom{r+s}{r-1} \binom{s+1}{2}\binom{r+s-2}{r-3}.\]
The other equalities follow from the same line of reasoning.
\end{proof}

\begin{lemma}\label{computArps} Let $f:V(G_{s+1,1,n+1})\rightarrow K^{\perp}$ be the normalized orthogonal projection and let $\lambda_1$ be the smallest eigenvalue of the adjacency matrix $M_{s+1,1,n+1}$. 
\begin{enumerate}
\item{ \[SD(f) = \frac{r+s}{2(s+1)(r-1)}\binom{r+s}{r-1}\binom{s+1}{2}\binom{r+s-2}{r-3}\] }
\item{ $M_{s+1,1,n+1}-\lambda_1I\succeq 0$ so $(-\lambda_1,\dots, -\lambda_1)$ is a feasible point for the dual problem and the value of the dual objective function at this point is
\[SD^{*}(\lambda_1) = \frac{1}{4}\binom{r+s}{r-1}\left(s\binom{r+s-1}{r-2}-\binom{r+s-2}{r-1}\right)\]}
\item{ The equality $SD(f)=SD^*(\lambda_1)$ holds. In particular $f$ is an optimal embedding of the multigraph $G_{s+1,1,n+1}$ into $K^{\perp}$.}
\item{ The expected value of a random cut obtained from this optimal embedding is
\[\mathbb{E}(W)=\frac{1}{2\pi}\binom{r+s}{r-1}\sum_{k=0}^s \binom{s+1}{k+1}\binom{r-1}{k+1}k\arccos\left(1-\frac{(r+s)(k+1)}{(s+1)(r-1)}\right)\]
}
\end{enumerate}
\end{lemma}
\begin{proof} $(1.)$ Follows from Lemma~\ref{evaluation} part $(1.)$ with  $a=s+1$,$b=1$, $c=n+2$ and $r=n-2$ and the fourth identity Lemma~\ref{combinatAr}. 
$(2.)$ From Lemma~\ref{spectrumArps} we know $\lambda_1=1-\binom{r+s-2}{r-2}$. The claim follows from Lemma~\ref{evaluation} part $(3.)$, the third identity in Lemma~\ref{combinatAr} and the equality $|V|=\binom{r+s}{r-1} $. $(3.)$ After cancelling common factors proving the equality $SD(f)=SD^*(\lambda_1)$ is equivalent to verifying
\[ (r+s)s\binom{r+s-2}{r-3} = (r-1)\left( s\binom{r+s-1}{r-2} - \binom{r+s-2}{r-1}\right).\]
Since $\binom{r+s-1}{r-2} = \frac{r+s-1}{r-2}\binom{r+s-2}{r-3}$ the above equality is equivalent to 
\[ -\frac{s(s+1)}{r-2}\binom{r+s-2}{r-3} = -(r-1)\binom{r+s-2}{r-1}\]
which is true. Optimality follows from weak duality.
$(4)$ The equality follows from Lemma~\ref{evaluation} part $(2.)$ and the first identity in Lemma~\ref{combinatAr}.
\end{proof}

\begin{lemma}\label{limitsArps} The following equality holds,
\[\lim_{r+s\rightarrow \infty}\frac{\lceil\frac{1}{2\pi}\binom{r+s}{r-1}\sum_{k=0}^s \binom{s+1}{k+1}\binom{r-1}{k+1}k\arccos\left(1-\frac{(r+s)(k+1)}{(s+1)(r-1)}\right)\rceil}{\frac{r+s}{2(s+1)(r-1)}\binom{r+s}{r-1}\binom{s+1}{2}\binom{r+s-2}{r-3}}=1\]
\end{lemma}

\begin{proof} Note that for $k\leq\frac{s(r-2)-1}{r+s}=k^*$
$$\frac{(r+s)(k+1)}{(s+1)(r-1)}\leq\frac{2}{\pi}\arccos\left(1-\frac{(r+s)(k+1)}{(s+1)(r-1)}\right)\leq 1$$
and for $k\geq\frac{s(r-2)-1}{r+s}$
$$\frac{(r+s)(k+1)}{(s+1)(r-1)}\geq\frac{2}{\pi}\arccos\left(1-\frac{(r+s)(k+1)}{(s+1)(r-1)}\right)\geq 1.$$
 Let $\gamma_{r,s}=\frac{2}{\pi}\sum_{k=0}^s \binom{s+1}{k+1}\binom{r-1}{k+1}k\arccos\left(1-\frac{(r+s)(k+1)}{(s+1)(r-1)}\right)$, $\delta_{r,s}=\sum_{k=0}^s \binom{s+1}{k+1}\binom{r-1}{k+1}k$ and $\beta_{r,s}=\sum_{k=0}^s \binom{s+1}{k+1}\binom{r-1}{k+1}k\frac{(r+s)(k+1)}{(s+1)(r-1)}$.
Hence
$$1-\frac{\beta_{r,s}^2-\delta_{r,s}^2}{\beta_{r,s}}\leq\frac{\gamma_{r,s}}{\beta_{r,s}}\leq 1.$$
with $\beta_{r,s}^2=\sum_{\lfloor k^*+1\rfloor}^{s} \binom{s+1}{k+1}\binom{r-1}{k+1}k\frac{(r+s)(k+1)}{(s+1)(r-1)}$. and $\delta_{r,s}^2=\sum_{k=\lfloor k^*+1\rfloor}^s \binom{s+1}{k+1}\binom{r-1}{k+1}k$. 
Now, for fix $r$, the expression on the left decreases as $s$ increases, then we will show the result for $s=r-2$. In this case $k^*=\frac{s-1}{2}$. Let $b(s)=\binom{s}{\lceil\frac{s}{2}\rceil}^2$, then by Lemma \ref{combinatAr} we have
$$\varphi(0):=\sum\limits_{k=\lfloor k^*+1\rfloor}^s \binom{s+1}{k+1}^2=\frac{1}{2}\binom{2(s+1)}{s+1}+O(b(s))$$
$$\varphi(1):=\sum\limits_{k=\lfloor k^*+1\rfloor}^s (k+1)\binom{s+1}{k+1}^2=\frac{1}{2}(s+1)\left(\binom{2s+1}{s}+O(b(s))\right)$$
and
$$\varphi(2):=\sum\limits_{k=\lfloor k^*+1\rfloor}^s k(k+1)\binom{s+1}{k+1}^2=\frac{1}{2}s(s+1)\left(\binom{2s}{s-1}+O(b(s))\right).$$ 
Hence
\begin{align*}
\beta^2_{s+2,s}-\delta^2_{s+2,s}&=\frac{2}{s+1}\varphi(2)-\varphi(1)+\varphi(0)\\
&=\binom{2(s+1)}{s+1}\frac{2s^2-3s-1}{4s^2}+sO(b(s)).
\end{align*}
Since $\beta_{s+2,s}=\binom{2(s+1)}{s+1}\frac{s^2}{2s+1}$, we have the desired result.
\end{proof}

\subsection{Type $D$ multigraphs.}
In this section we study the multigraphs $C_{2,2,n}$. Recall that $r=n+2$. We begin by some elementary combinatorial lemmas.

\begin{lemma} \label{combinatevensubs}  Let $E(r)$ and $O(r)$ denote the subsets of $[r]$ of even and odd size respectively. The following statements hold:
\begin{enumerate} 
\item{\label{implication} Let $j$ be a nonnegative integer. If for $0\leq s\leq j$ we have
\[\sum_{S\in E(r-1)} |S|^s=\sum_{S\in O(r-1)}|S|^s\] then the same equalities holds when the sums run over subsets of $[r]$.}
\item{\label{eq} For $r\geq 3$ and $j=0,1,2$ we have
\[\sum_{S\in E(r)} |S|^j=\sum_{S\in O(r)}|S|^j\]
}
\item{\label{eq2} For $r\geq 3$ we have
\[\begin{array}{l}
\sum_{S\in E(r)} 1= 2^{r-1}\\
\sum_{S\in E(r)}|S| = r2^{r-2}\\
\sum_{S\in E(r)}|S|^2 = r(r-1)2^{r-3} + r2^{r-2}\\
\end{array}
\]
}
\end{enumerate}
\end{lemma}
\begin{proof} (\ref{implication}) The following equalities hold:
\[ \sum_{S\in E(r)} |S|^j = \sum_{S\in E(r), S\ni r} |S|^j + \sum_{S\in E(r),S\not\ni r} |S|^j = \sum_{S=S'\cup \{r\},S'\in O(r-1)} (|S'|+1)^j + \sum_{S\in E(r-1)} |S|^j\]
by our assumptions the latter equals
\[\sum_{S=S'\cup \{r\},S'\in E(r-1)} (|S'|+1)^j + \sum_{S\in O(r-1)} |S|^j = \sum_{S\in O(r)} |S|^j.\]
(\ref{eq}) When $r=3$ the equalities hold for $j=0,1,2$ by direct calculation. The claim follows immediately from part (\ref{implication}). (\ref{eq2}) By part (\ref{eq}) it suffices to divide by two the sum over all subsets of $[r]$. Now,
\[
\sum_{S\subseteq r} |S|^j = 
\begin{cases} 
2^{r} & \text{if $j=0$}\\
r2^{r-1} & \text{if $j=1$}\\
r(r-1)2^{r-2} + r2^{r-1} & \text{if $j=3$}\\
\end{cases}
\]
and the claim follows. \end{proof}
Next, assume $r\geq 5$ and let $M$ be the adjacency matrix of $G_{2,2,n}$. Define $B:=M-I$
\begin{lemma} \label{spectrumDr} The matrix $B$ satisfies the equation 
$B^2=(q - 2^{r-3})J - 2^{r-3}B$ where $q:=(r^2-7r+16)2^{r-5}$.
As a result $C_{2,2,n}$ is a strongly regular multigraph and its spectrum consists of 
$1-2^{r-3} < 1 < 1+(r-4)2^{r-3}$ with multiplicities $r$, $2^{r-1}-(r+1)$ and one respectively.
\end{lemma}
\begin{proof} Let $q:=B^2_{\emptyset\emptyset}=\sum_{S\subseteq E(r)} \left(\frac{|S|}{2}-1\right)^2$. By Lemma~\ref{combinatevensubs} this quantity equals $(r^2-7r+16)2^{r-5}$. 
Since the Weyl group of $D_r$ acts transitively on the multigraph $C_{2,2,r}$ to prove that $B$ satisfies the above quadratic equation it suffices to show that for every nonempty $T\subseteq [r]$ with $|T|=2t$ the equality $B^{2}_{\emptyset T} = q+2^{r-3}-2^{r-3}B_{\emptyset T}$ holds. 

Define $E_T(k)=\{S\in E(r): |S\cap T|=k\}$ and note that $B^2_{\emptyset T}$ equals
\[\sum_{S\in E(r)} B_{\emptyset S}B_{ST} = \sum_{k=0}^{2t}\sum_{S\in E_T(k)} B_{\emptyset S}B_{ST} =\sum_{k=0}^{2t}\sum_{S\in E_T(k)} \left(\frac{|S|}{2}-1\right)\left(\frac{|S|+2(t-k)}{2}-1\right)\]
which equals
\[q +\sum_{k=0}^{t-1} (t-k)\left( \sum_{S\in E_T(k)}\left(\frac{|S|}{2}-1\right)-\sum_{S\in E_T(2t-k)}\left(\frac{|S|}{2}-1\right)\right)\]
now, replacing $T\cap S$ with its complement in $T$ determines a bijection between $E_T(k)$ and $E_T(2t-k)$ so the last term equals
\[ q -\sum_{k=0}^{t-1} (t-k)^2 |E_T(k)| = q - \sum_{k=0}^{t-1} (t-k)^2\binom{2t}{k}2^{r-2t-1}= q -2^{r-2t-2}\sum_{k=0}^{2t}(t-k)^2\binom{2t}{k}\]
Finally, the equality $\sum_{k=0}^{2t} \binom{2t}{k}(t-k)^2 = t2^{2t-1}$ implies that
\[  q -2^{r-2t-2}\sum_{k=0}^{2t}(t-k)^2\binom{2t}{k} = q - t2^{r-3}\] 
now $t=B_{\emptyset T} +1$ and thus $B$ satisfies the matrix equation claimed above. 
To find the spectrum argue as in the proof of Theorem~\ref{multigraphSpectra} using the quadratic equation. 
\end{proof}

\begin{lemma}\label{computDr} Let $f:V(G_{2,2,n})\rightarrow K^{\perp}$ be the normalized orthogonal projection and let $\lambda_1$ be the smallest eigenvalue of the adjacency matrix $M$. 
\begin{enumerate}
\item{$SD(f) = (r-3)2^{2r-6}$ }
\item{ $M_{2,2,n}-\lambda_1I\succeq 0$ so $(-\lambda_1,\dots, -\lambda_1)$ is a feasible point for the dual problem and the value of the dual objective function at this point is
$SD^{*}(\lambda_1) = (r-3)2^{2r-6}$.}
\item{ The equality $SD(f)=SD^*(\lambda_1)$ holds. In particular $f$ is an optimal embedding of the multigraph $C_{2,2,n}$ into $K^{\perp}$.}
\item{ The expected value of a random cut obtained from this optimal embedding is
\[\mathbb{E}(W)=\frac{2^{r-2}}{\pi}\sum_{k=0}^{\lfloor\frac{r}{2}\rfloor-1}k\arccos\left(1-\frac{4(k+1)}{r}\right)\]
}
\end{enumerate}
\end{lemma}
\begin{proof} Recall that $S_k(2,2,n)$ is the number of edges of weight $k$ in $C_{2,2,n}$. By transitivity of the action of the Weyl group of $D_r$ on $C_{2,2,n}$ the number of edges of weight $k$ incident to every vertex $V$ is the same and is easily seen to be $\binom{r}{2(k+1)}$ when $V=\emptyset$. Summing over all exceptional classes we see that
\[2S_k=2^{r-1}\binom{r}{2(k+1)}.\] 
From this identity together with Lemma~\ref{evaluation} part $(1.)$ applied with  $a=b=2$, $c=n$ and $r=n+2$  we have 
\[SD(f)=2^{r-3} \sum_{k=0}^{\lfloor \frac{r}{2}\rfloor -1} \binom{r}{2(k+1)}\frac{4(k+1)k}{r} = (r-3)2^{2r-6}\]
and the last equality follows from Lemma~\ref{combinatevensubs}.  $(2.)$  From Lemma~\ref{spectrumDr} we know $\lambda_1=1-2^{r-3}$ As a result by Lemma~\ref{evaluation} part $(3.)$ we have
\[SD^*(-\lambda_1)=2^{r-3} \left(\sum_{k=0}^{\lfloor \frac{r}{2}\rfloor -1} \binom{r}{2(k+1)}k + 2^{r-3}-1\right) = (r-3)2^{2r-6}\]
where the last equality follows from Lemma~\ref{combinatevensubs}. $(3.)$ Optimality follows from weak duality.
$(4)$ The claim follows from Lemma~\ref{evaluation} part $(2.)$ and the above expression for $S_k$.
\end{proof}
Finally we show that the essential performance ratio of the graphs $C_{2,2,n}$ is equal to one. To this end we need the following combinatorial Lemma,
\\

\begin{lemma}\label{phi}
Let $b(r)=\binom{r}{\lceil\frac{r}{2}\rceil}$. For $r\geq5$ the following statements hold:
\begin{enumerate}
\item \label{sets}
\[\begin{array}{l}
\sum\limits_{S\in E(r),|S|\geq \frac{r}{2}} 1= 2^{r-2}+O(b(r))\\
\sum\limits_{S\in E(r),|S|\geq \frac{r}{2}}|S| = r(2^{r-3}+O(b(r)))\\
\sum\limits_{S\in E(r),|S|\geq \frac{r}{2}}|S|(|S|-1) = r(r-1)(2^{r-4} + O(b(r)))\\
\end{array}
\]
\item \label{eqphi}$\sum\limits_{k=\lceil\frac{r}{4}\rceil-1}^{\lfloor \frac{r}{2}\rfloor -1} \binom{r}{2(k+1)}k\left(\frac{4(k+1)}{r}-1\right)=O(2^r)$.
\end{enumerate}
Here $O(\cdot)$ refers to the big-O notation and not to number of odd sets.
\end{lemma}
\begin{proof}
(\ref{sets}) Follows from Lemma \ref{combinatevensubs} (\ref{eq2}). (\ref{eqphi}) Let 
$$\phi(0):=\sum\limits_{k=\lceil\frac{r}{4}\rceil-1}^{\lfloor \frac{r}{2}\rfloor -1} \binom{r}{2(k+1)}=\sum\limits_{S\in E(r),|S|\geq \frac{r}{2}} 1,$$
$$\phi(1):=\sum\limits_{k=\lceil\frac{r}{4}\rceil-1}^{\lfloor \frac{r}{2}\rfloor -1} 2(k+1)\binom{r}{2(k+1)}=\sum\limits_{S\in E(r),|S|\geq \frac{r}{2}} |S|$$
and
$$\phi(2):=\sum\limits_{k=\lceil\frac{r}{4}\rceil-1}^{\lfloor \frac{r}{2}\rfloor -1} 2(k+1)(2k+1)\binom{r}{2(k+1)}=\sum\limits_{S\in E(r),|S|\geq \frac{r}{2}}|S|(|S|-1).$$ 
Hence 
\begin{align*}
\sum\limits_{k=\lceil\frac{r}{4}\rceil-1}^{\lfloor \frac{r}{2}\rfloor -1} \binom{r}{2(k+1)}&k\left(\frac{4(k+1)}{r}-1\right)\\
&=\frac{4}{r}\left(\frac{\phi(2)-3\phi(1)+4\phi(0)}{4}\right)+\left(\frac{4}{r}-1\right)\left(\frac{\phi(1)-2\phi(0)}{2}\right)\\
&=O(2^r)+O(b(r))=O(2^r).
\end{align*}
\end{proof}

\begin{lemma}\label{limitsDr} The following equality holds,
\[\lim_{r\rightarrow \infty}\frac{\lceil\frac{2^{r-2}}{\pi} \sum_{k=0}^{\lfloor \frac{r}{2}\rfloor-1} k\binom{r}{2(k+1)} \arccos(1-\frac{4(k+1)}{r})\rceil}{(r-3)2^{2r-6}}=1\]
\end{lemma}
\begin{proof} Let $\gamma_r$ be the numerator in the expression and $\beta_r$ the denominator. Note that for $k<\frac{r}{4}-1$
$$\frac{4(k+1)}{r}\leq\frac{2}{\pi}\arccos\left(1-\frac{4(k+1)}{r}\right)\leq 1$$
and for $k\geq\frac{r}{4}-1$
$$\frac{4(k+1)}{r}\geq\frac{2}{\pi}\arccos\left(1-\frac{4(k+1)}{r}\right)\geq 1.$$
Hence, by Lemma \ref{computDr}
$$\gamma_r+\beta_r^2\geq\beta_r+\delta_r^2,$$
with $\beta_r^2=2^{r-3} \sum\limits_{k=\lceil\frac{r}{4}\rceil-1}^{\lfloor \frac{r}{2}\rfloor -1} \binom{r}{2(k+1)}\frac{4(k+1)k}{r}$ and $\delta_r^2=2^{r-3} \sum\limits_{k=\lceil\frac{r}{4}\rceil-1}^{\lfloor \frac{r}{2}\rfloor -1} \binom{r}{2(k+1)}k$. Therefore
$$1-\frac{\beta_r^2-\delta_r^2}{\beta_r}\leq\frac{\gamma_r}{\beta_r}\leq 1,$$
and the result follows from Lemma \ref{phi}.
\end{proof}

\subsection{Exceptional multigraphs.}\label{Es}

In this section we focus in the two remaining minuscule graphs $G_{2,3,3}$ and $G_{2,4,3}$. These are the graphs of exceptional curves on del Pezzo surfaces of degrees $3$ and $2$.  Since contracting any of the exceptional curves of these surfaces brings us to a Del Pezzo surface of a larger degree is easy to compute the numbers $S_k(a,b,c)$  as in the table below. The spectrum can be computed directly and shown to have three values $\lambda_1<1<\eta_1$ as in the table below,
\[
\begin{array}{llllll}
T_{a,b,c} & |V| & S_1 & S_2 & \lambda_1 & \eta_1\\ 
\hline
T_{2,3,3} & 27 & 135 & 0 & -5 & 10\\
T_{2,4,3} & 56 & 756 & 28 & -11 & 29\\
\end{array}
\]
as a result we have
\begin{lemma}\label{computEs}  Let $(a,b,c)$ be either one of $(2,3,3)$ or $(2,4,3)$ and let $f:G_{a,b,c}\rightarrow K^{\perp}$ be the normalized orthogonal projection and let $\lambda_1$ be the smallest eigenvalue of $M_{a,b,c}$. Then
\begin{enumerate}
\item{ The values of $SD(f)$ and $SD^{*}(-\lambda_1)$ are given in the following table
\[
\begin{array}{l|l|l}
T_{a,b,c}  & SD(f) & SD^*(\lambda_1)\\
\hline
T_{2,3,3} & \frac{3}{4}135=\frac{405}{4} & \frac{1}{2}135-\frac{27}{4}(-5)=\frac{405}{4} \\
T_{2,4,3} &  \frac{2}{3}(756+3\cdot 28)=560 & \frac{1}{2}(756+2\cdot 28) -\frac{56}{4}(-11)=560\\
\end{array}
\]
in particular, the embedding $f$ is optimal for the geometric relaxation of maxcut.
}
\item{ The expected value of a random cut obtained from this optimal embedding is
\[
\begin{array}{l|l}
T_{a,b,c} & \mathbb{E}[W]\\ 
\hline
T_{2,3,3} & \frac{1}{\pi} 135 \arccos(1-\frac{3}{4}\cdot 2)=90\\
T_{2,4,3} & \lceil\frac{1}{\pi} (756 \arccos(1-\frac{2}{3}\cdot 2) + 56\arccos(1-\frac{2}{3}\cdot 3))\rceil = 516\\
\end{array}
\]
}
\end{enumerate}
\end{lemma}
\begin{proof} The results follow from Lemma~\ref{evaluation} using the information contained in the first table of this section.
\end{proof}

\subsection{Stochastic simulation of divisor cuts.}\label{Simulations}
Theorem~\ref{main} gives us a complete description of the optimal embedding $f$. Using this known optimal embedding random hyperplane cuts can be simulated very efficiently.
We carried out these simulations on a computer for several values of the parameters and summarize our results in the tables below for the infinite families of type $A$ and $D$ respectively. The number on top shows the mean value of the cuts, the second below shows the maximum cut found in the simulations, the third one shows the value of the semidefinite relaxation and the bottom one the variation coefficient.

\begin{table}[h]
\caption{Simulations of type $A$ family}
\begin{small}
        \centering
                \begin{tabular}{|c|c|c|c|c|c|}
                        \hline
		   $(r,s)$ & 1 & 2 & 3 & 4 & 5\\
		  \hline
		   4 & $\begin{array}{c} 10.986\\12\\12.5\\0.071\end{array}$ & $\begin{array}{c}74.736\\\bf{80}\\\bf{80}\\0.038\end{array}$  & - & - & -\\
		  \hline
		  5 & $\begin{array}{c}30\\30\\33.75\\0\end{array}$ & $\begin{array}{c}282.064\\300\\300.25\\0.024\end{array}$ & $\begin{array}{c}1418.7\\\bf{1575}\\\bf{1575}\\0.015\end{array}$ & - & -\\
		  \hline
		  6 & $\begin{array}{c}66.246\\70\\73.5\\0.027\end{array}$ & $\begin{array}{c}832.53\\850\\896\\0.009\end{array}$ & $\begin{array}{c}5614.4\\5880\\5953.5\\0.009\end{array}$ & $\begin{array}{c}26872\\\bf{28224}\\\bf{28224}\\0.007\end{array}$ & -\\
		  \hline
		  7 & $\begin{array}{c}127.71\\135\\140\\0.025\end{array}$ & $\begin{array}{c}2065.1\\2095\\2205\\0.005\end{array}$ & $\begin{array}{c}17433\\17775\\18375\\0.004\end{array}$ & $\begin{array}{c}102024\\105840\\106722\\0.004\end{array}$ & $\begin{array}{c}466219\\\bf{485100}\\\bf{485100}\\0.004\end{array}$\\
		  \hline
		  8 & $\begin{array}{c}223.84\\231\\243\\0.02\end{array}$ & $\begin{array}{c}4524.8\\4648\\4800\\0.007\end{array}$ & $\begin{array}{c}46723\\47215\\49005\\0.002\end{array}$ & $\begin{array}{c}328424\\334474\\342144\\0.003\end{array}$ & \\
		  \hline
 		 9 & $\begin{array}{c}365.66\\378\\393.75\\0.017\end{array}$ & $\begin{array}{c}9031.8\\9324\\9528.7\\0.007\end{array}$ & $\begin{array}{c}111959\\113246\\116944\\0.002\end{array}$  &  & \\
		  \hline
                \end{tabular}
\end{small}
\label{tableA}
\end{table} 

\begin{table}[h]
\caption{Simulations of type $D$ family}
        \centering
                \begin{tabular}{|c|c|c|c|c|c|c|}
		\hline
		$r$ & 5 & 6 & 7 & 8 & 9 & 10\\
		\hline
		 & $\begin{array}{c}28.191\\\bf{32}\\\bf{32}\\0.064\end{array}$ & $\begin{array}{c}177.968\\\bf{192}\\\bf{192}\\0.027\end{array}$ & $\begin{array}{c}948.58\\\bf{1024}\\\bf{1024}\\0.019\end{array}$ & $\begin{array}{c}4821.34\\\bf{5120}\\\bf{5120}\\0.011\end{array}$ & $\begin{array}{c}23278.8\\\bf{24576}\\\bf{24576}\\0.008\end{array}$ & $\begin{array}{c}109472\\\bf{114688}\\\bf{114688}\\0.005\end{array}$\\
		\hline
                \end{tabular}
\label{tableD}
\end{table}

\subsection{Proof of Theorem~\ref{main}.}\label{proofofMainTheorem}
\begin{proof} (\ref{spectrumTabc}.) Follows from Lemmas~\ref{spectrumArps}, ~\ref{spectrumDr} and the first table in Section~\ref{Es}. (\ref{optimalityTabc}.) The claim is proven in Lemmas~\ref{computArps}, ~\ref{computDr} and ~\ref{computEs}.  (\ref{boundsTabc}.) Recall that for the optimal embedding $f: V(G_{a,b,c})\rightarrow K^{\perp}$ we have
\[\mathbb{E}[W(f)]\leq m({a,b,c})\leq SD(f)\]
and since $m(a,b,c)$ is an integer the inequalities can be improved by adding integer parts on both sides. The values $\ell(a,b,c)$ and $u(a,b,c)$ have been computed in Lemmas~\ref{computArps}, ~\ref{computDr} and ~\ref{computEs}. Moreover from the above inequalities it follows that
\[ \frac{\mathbb{E}[W(f)]}{SD(f)}\leq \frac{m(a,b,c)}{SD(f)}\leq 1\]
and the leftmost quantity converges to one for infinite families, as shown in Lemmas~\ref{limitsArps} and ~\ref{limitsDr}.
Claim $(4)$ follows from the random hyperplane cuts shown in the previous section. We believe these equalities hold in general.
\end{proof}

\end{document}